\documentclass[12pt,a4paper]{amsart}
\usepackage{amsfonts,amssymb,amscd,amsmath,enumerate,verbatim,calc}
\usepackage{mathrsfs}
\numberwithin{equation}{section}
\newtheorem{thm}{Theorem}[section]
\newtheorem{cor}[thm]{Corollary}
\newtheorem{lem}[thm]{Lemma}

\newtheorem{defn}[thm]{Definition}

\newtheorem{exam}[thm]{Example}
\newtheorem{rem}[thm]{Remark}

\DeclareMathOperator{\Ext}{Ext} \DeclareMathOperator{\Supp}{Supp}
\DeclareMathOperator{\V}{V} \DeclareMathOperator{\Hom}{Hom}
\DeclareMathOperator{\Ker}{Ker} \DeclareMathOperator{\Coker}{Coker}
\DeclareMathOperator{\Image}{Im}

 \DeclareMathOperator{\Max}{Max}
\DeclareMathOperator{\lc}{H} 
 
\DeclareMathOperator{\G}{\Gamma} 
\DeclareMathOperator{\E}{E}

\DeclareMathOperator{\Ass}{Ass}

\DeclareMathOperator{\dimSupp}{dimSupp}

\newcommand{\oops}[1]{%
	\textbf{#1}}
\newcommand{\fa}{\mathfrak{a}}

\newcommand{\fm}{\mathfrak{m}}
\newcommand{\fp}{\mathfrak{p}}

\newcommand{\lo}{\longrightarrow}

\begin{document}

\title[Weakly cofiniteness of local
cohomology modules]
{Weakly cofiniteness of  local
	cohomology modules}

\bibliographystyle{amsplain}

   \author[M. Aghapournahr]{Moharram Aghapournahr}
\address{Department of Mathematics, Faculty of Science, Arak University,
Arak, 38156-8-8349, Iran.}
\email{m-aghapour@araku.ac.ir}






\begin{abstract}
Let $R$ be a commutative Noetherian ring, $\Phi$ a system of ideals of $R$
and $I\in \Phi$.
Let $M$ be an $R$-module (not necessary $I$-torsion) such that  $\dim M\leq 1$, then the $R$-module $\Ext^i_{R}(R/I,
M)$ is weakly Laskerian, for all $i\geq 0$, if and only if the $R$-module
$\Ext^i_{R}(R/I, M)$ is weakly Laskerian, for $i=0, 1$. Let $t\in\Bbb{N}_0$ be an integer and $M$ an $R$-module such that $\Ext^i_R(R/I,M)$ is weakly Laskerian for all $i\leq t+1$. We prove that if the $R$-module 
$\lc^{i}_\Phi(M)$ is
${\rm FD_{\leq 1}}$ for all $i<t$, then $\lc^{i}_\Phi(M)$ is
$\Phi$-weakly cofinite for all $i<t$ and for any ${\rm FD_{\leq 0}}$ (or minimax)
submodule $N$ of $\lc^t_\Phi(M)$, the $R$-modules
$\Hom_R(R/I,\lc^t_\Phi(M)/N)$  and $\Ext^1_R(R/I,\lc^t_\Phi(M)/N)$ are weakly Laskerian. Let $N$ be a finitely generated $R$-module. We also prove that $\Ext^j_R(N,\lc^{i}_\Phi(M))$ and ${\rm Tor}^R_{j}(N,H^{i}_\Phi(M))$ are $\Phi$-weakly cofinite for all $i$ and $j$ whenever $M$ is weakly Laskerian and $\lc^{i}_\Phi(M)$ is ${\rm
	FD_{\leq 1}}$ for all $i$. Similar results are true for ordinary local cohomology modules and local cohomology modules defined by a pair of ideals. 
\end{abstract}

\maketitle


\section{Introduction}


Throughout  this paper $R$ is a commutative Noetherian ring with non-zero
identity and $I$ an ideal of $R$. For an $R$-module $M$,  the $i^{th}$
local cohomology module $M$ with respect to ideal $I$ is defined as
\begin{center}
	$\lc^{i}_{I}(M) \cong \underset{n}\varinjlim \Ext^{i}_{R}(R/{I}^{n},M).$
\end{center}

Grothendieck in \cite{Gro} posed the following
conjecture:

\noindent {\bf Conjecture 1.1.}\label{1.1} {\it Let $M$ be a finitely generated $R$--module and
	$I$ an ideal of $R$. Then $\Hom_R(R/I , \lc^{i}_{I}(M))$
	is finite for all $i\geq
	0$}.

This conjecture is not true in general as Hartshorne  showed in \cite{Har}, but some  authors proved that for
some number $t$, the module
$\Hom_{R}(R/I,H^t_I(M))$ is finite under some
conditions. See \cite[Theorem 3.3]{AKS}, \cite[Theorem
6.3.9]{DY2}, \cite[Theorem 2.1]{DY}, \cite[Theorem 2.6]{BN}, \cite[Theorem
2.3]{BN2} and \cite[Theorem 3.4]{AB}. Hartshorne   also  defined a module $M$
to be \emph{$I$--cofinite} if $\Supp_R(M)\subseteq \V(I)$ and
$\Ext^{i}_{R}(R/I,M)$ is finitely generated for all $i\geq0$ and posed the following
question:

\noindent{\bf Question 1.2.}\label{1.2} \emph{ Let $M$ be a finite $R$--module
	and $I$ be an ideal of $R$. When are $ \lc^{i}_{I}(M)$ $I$--cofinite for all $i\geq
	0$?}

This question was studied  by several authors in  \cite{Har, HK, DM, MV, Mel2, Mel, BN} and  \cite{AB}.

There are some generalizations of the theory of ordinary local cohomology modules.
The following is introduced by Bijan-Zadeh in \cite{BZ1}.

Let $\Phi$ be a non-empty set of ideals of $R$. We call $\Phi$ a {\it system of
	ideals} of $R$ if, whenever $I_1,I_2\in\Phi $, then there is an ideal $J\in\Phi$
such that $J\subseteq I_1{I_2}$. For such a system, for every $R$-module $M$,
one can define
$$\G_{\Phi}(M)=\{~x\in M\mid Ix=0 \textmd{\ \ for some}~ I\in \Phi\}.$$

Then $\G_{\Phi}(-)$ is a functor from $\mathscr{C}(R)$ to itself (where
$\mathscr{C}(R)$ denotes the category of all $R$-modules and all
$R$-homomorphisms). The functor  $\G_{\Phi}(-)$ is additive, covariant,
$R$-linear and left exact. In \cite{BZ2}, $\G_{\Phi}(-)$ is denoted by $L_{\Phi}(-)$
and is called the \textquotedblleft general local cohomology functor with respect to $\Phi$\textquotedblright. For
each $i\geq 0$, the $i$-th right derived functor of $\G_{\Phi}(-)$ is denoted by
$\lc_{\Phi}^i(-)$. The functor $\lc_{\Phi}^i(-)$ and $\underset{I\in
	\Phi}\varinjlim\lc_{I}^i(-)$ (from $\mathscr{C}(R)$ to itself) are naturally equivalent
(see \cite{BZ1}). For an ideal $I$ of $R$, if $\Phi=\{I^n| n\in \mathbb{N}_0\}$, then
the functor $\lc_{\Phi}^i(-)$ coincides with the ordinary local cohomology functor
$\lc_{I}^i(-)$.
\oops{It is shown that, the study of torsion theory over $R$ is equivalent to study the
general local cohomology theory (see \cite{BZ2}).}

As a special case of  \cite[Definition 2.1]{Y} and generalization of FSF
modules (see \cite[Definition 2.1]{hung}), in \cite[Definition 2.1]{AB} the author of present paper and Bahmanpour
introduced the class of {\it{${\rm FD_{\leq n}}$}} modules. A module
$M$ is said to be ${\rm FD_{\leq n}}$ module, if there exists a finitely generated
submodule  $N$ of $M$ such that $\dim M/N\leq n$. For more details about
properties of this class see \cite[Lemma 2.3]{AB}. Note that the class of ${\rm
	FD_{\leq -1}}$ is the same as finitely generated $R$-modules.
Recall that a module $M$ is a \emph{minimax} module if there is a
finitely generated submodule $N$ of $M$ such that the quotient module $M/N$
is Artinian.  Minimax modules have been studied by
Z\"{o}schinger in
\cite{Zrmm}.   
Recall too that an
$R$-module $M$ is called \emph{weakly Laskerian} if $\Ass_R(M/N)$ is a
finite set for each  submodule $N$ of $M$. The class of weakly Laskerian
modules was introduced in \cite{DiM} by Divaani-Aazar and Mafi. They also as a generalization of cofinite modules with respect to an ideal in \cite{DiM2}  defined an $R$-module $M$ to be  \emph{weakly cofinite with respect to ideal $I$ of $R$} or $I$-weakly cofinite if $\Supp_R(M)\subseteq \V(I)$ and
$\Ext^{i}_{R}(R/I,M)$ is weakly Laskerian for all $i\geq0$. In \cite[Definition 4.2]{AB1} the author of present paper and Bahmanpour introduced  the concept of
{\it$\Phi$-weakly cofiniteness} of general local cohomology  modules. The general local 
cohomology module $\lc_{\Phi}^j(M)$ is defined to be $\Phi$-weakly cofinite if there exists
an ideal $I \in \Phi$ such that $\Ext^i_R(R/I,\lc_{\Phi}^j(M))$ is  weakly Laskerian, for all
$i,j\geq0$.

Recently many authors studied the weakly Laskerianness and weakly cofiniteness of local cohomology modules and answered the Conjecture 1.1 and  Question 1.2 in the class of weakly Laskerian modules in some cases (see  \cite{DiM, DiM2, AM, Z, NT, VA, AB,  BNS2}). The purpose of this note is to make a suitable
generalization of Conjecture 1.1 and Question 1.2 in terms of weakly Laskerian modules instead of finitely generated modules for general local cohomology modules. In this direction in Section
2, we generalize \cite[Theorem 3.4 and Corollaries 3.5 and 3.6]{AB} and \cite[Theorem 2.9 and Corollaries 2.10]{AB1}. More precisely, we shall show that:\\

\noindent{\bf Theorem 1.3.} (See Theorem \ref{homext}) {\it Let $R$ be a Noetherian ring and $I\in \Phi$ an ideal of $R$. Let $t\in\Bbb{N}_0$ be an integer and $M$ an $R$-module such that $\Ext^i_R(R/I,M)$ are weakly Laskerian for all $i\leq t+1$. Let the $R$-modules $\lc^{i}_\Phi(M)$ are ${\rm FD_{\leq 1}}$ $R$-modules for all $i<t$. Then, the following conditions hold:
	\begin{itemize}
		\item[(i)]  The $R$-modules $\lc^{i}_\Phi(M)$ are $I$-$ETH$-weakly cofinite {\rm{(}}in particular $\Phi$-weakly cofinite{\rm{)}} for all $i<t$.
		\item[(ii)]  For all ${\rm FD_{\leq 0}}$ {\rm{(}}or minimax{\rm{)}} submodule $N$ of
		$\lc^{t}_\Phi(M)$, the $R$-modules $${\rm Hom}_R(R/I,\lc^{t}_\Phi(M)/N)\,\,\,
		{\rm and}\,\,\,{\rm Ext}^1_R(R/I,\lc^{t}_\Phi(M)/N)$$ are weakly Laskerian.
	\end{itemize}}

	\noindent{\bf Corrolary 1.4.} (See Corollary \ref{cof}) {\it  Let $R$ be a Noetherian ring and $I\in \Phi$ an ideal of $R$. Let $M$
		be an $I$-$ETH$-weakly cofinite $R$-module  such that the $R$-modules $\lc^{i}_\Phi(M)$
		are ${\rm FD_{\leq 1}}$ $R$-modules for all
		$i$. Then,
		\begin{itemize}
			\item[(i)] the $R$-modules $\lc^{i}_\Phi(M)$ are $I$-$ETH$-weakly cofinite {\rm{(}}in particular, $\Phi$-weakly cofinite{\rm{)}} for all $i$.
			\item[(ii)]  for any $i\geq 0$ and for any ${\rm FD_{\leq 0}}$ {\rm{(}}or minimax{\rm{)}} submodule $N$ of
			$\lc^{i}_\Phi(M)$, the $R$-module $\lc^{i}_\Phi(M)/N$ is   $I$-$ETH$-weakly cofinite
			{\rm{(}}in particular, $\Phi$-weakly cofinite{\rm{)}}.
		\end{itemize}}

		Hartshorne also asked the following question:
		
		\noindent{\bf Question 1.5.} \emph{Whether the category $\mathscr{M}(R,
			I)_{cof}$ of $I$-cofinite modules forms an Abelian subcategory of the category of
			all $R$-modules? That is, if $f: M\longrightarrow N$ is an $R$-module
			homomorphism of $I$-cofinite modules, are $\Ker f$ and $\Coker f$ $I$-cofinite?}
		
		With respect to this question, Hartshorne showed that if $I$ is a prime ideal of
		dimension one in a complete regular local ring $R$, then the answer to his
		question is positive. On the other hand, in \cite{DM}, Delfino and Marley extended
		this result to arbitrary complete local rings. Recently, Kawasaki \cite{Ka2}
		generalized the Delfino and Marley's result for an arbitrary ideal $I$ of dimension
		one in a local ring $R$. Finally,  Melkersson  in \cite{Mel1} completely have
		removed local assumption on $R$. More recently, in \cite{BNS} (resp. \cite{AB})
		it is shown that Hartshorne's question is true for  the category of all $I$-cofinite
		$R$-modules $M$ with $\dim M\leq 1$ (resp. the class of $I$- cofinite~${\rm
			FD_{\leq1}}$ modules),  for all ideals $I$ in a commutative Noetherian ring $R$. Also in \cite{BNS2} 
		it is proved that the same question is true for  the category of all $I$-weakly cofinite
		$R$-modules $M$ with $\dim M\leq 1$ for all ideals $I$ in  $R$. In this direction we introduced the concept of {\it $I$-$ETH$-weakly cofinite} or $ETH$-weakly cofinte modules with respect to $I$ in Definition \ref{def}. One of the main results of this
		section is to prove that the class of $I$-$ETH$-weakly cofinite and ${\rm FD_{\leq
				1}}$($\mathscr {FD}^1(R,I)_{ethwcof}$) modules are  Abelian category (see
		Theorem \ref{abel1}). Using this fact we prove the following corollary:\\
		
		\noindent{\bf Corrolary 1.6.} (See Corollary \ref{torextlc}) {\it Let $\Phi$ be a system of ideals of  a Noetherian ring $R$, $M$ a non-zero
			$I$-$ETH$-weakly cofinite $R$-module such that $ \lc^{i}_{\Phi}(M)$ are ${\rm FD_{\leq 1}}$~ $R$-modules for all $ i \geq 0$. Then for each finite
			$R$-module $N$, the $R$-modules ${\rm Ext}^{j}_R(N,\lc^{i}_{\Phi}(M))$ and ${\rm
				Tor}^R_{j}(N,\lc^{i}_{\Phi}(M))$ are $\Phi$-weakly cofinite and ${\rm FD_{\leq 1}}$~ $R$-modules for all $ i \geq 0$ and $ j \geq 0$.}\\
		
		In Section 3 we prove that similar corollaries are true for local cohomology modules defined  by a pair of
		ideals because it is a special case of local cohomology with respect to a
		system of ideals.

		Throughout this paper, $R$ will always be a commutative Noetherian ring with
		non-zero identity and $I$ will be an ideal of $R$.  We denote $\{\frak p \in {\rm
			Spec}\,R:\, \frak p\supseteq I \}$ by $V(I)$. 
		For any unexplained notation and terminology we refer the reader to \cite{BSh} and \cite{BH}.




\section{$ETH$-weakly cofinite modules with respect to an ideal}


The definitions of $ETH$-cofinite module and weakly cofinite module with respect to an ideal (\cite[Definitions
2.2]{A} and \cite[Definition 2.4]{DiM2}), motivate the following definition.

\begin{defn}\label{def}
	An  $R$-module $M$ {\rm{(}}not necessary $I$-torsion{\rm{)}} is called
	$ETH$-weakly cofinite with respect to an ideal $I$ of $R$ or $I$-$ETH$-weakly cofinite if
	$\Ext^{i}_{R}(R/I,M)$ is a weakly Laskerian $R$-module for all $i$.
\end{defn}

\begin{rem} Let $I$ be an  ideal of $R$.
	\begin{itemize}\item[(i)]  All weakly Laskerian $R$-modules, $ETH$-cofinite and weakly cofinite $R$-modules with respect to ideal $I$  are $I$-$ETH$-weakly cofinite.
		\item[(ii)]  Suppose $M$ is an $I$-torsion  module,
		then $M$ is $I$-$ETH$-weakly cofinite if and only if it is $I$-weakly cofinite module.
	\end{itemize}	
\end{rem}


We claim that the class of $ETH$-weakly cofinite modules with respect to an ideal  is
strictly larger than the class of $ETH$-cofinite and weakly cofinite modules with respect to the same ideal. To do this, see the following examples.

\begin{exam} 
{\rm{(}}i{\rm{)}}	Let $(R, \fm)$ be a Noetherian local ring of dimension $d>0$. Let
	$M=R\oplus E( R/\fm)$. It is easy to see that $M$ is an $\fm$-$ETH$-weakly cofinite
	$R$-module  that is not  $\fm$-cofinite.
	
{\rm{(}}ii{\rm{)}} Let $(R, \fm)$ be a Noetherian local ring of dimension $d>0$. Let
$M=R\oplus(\underset{i\in \Bbb N}\oplus  R/\fm)$. It is easy to see that $M$ is an $\fm$-$ETH$-weakly cofinite
$R$-module  that is not  $\fm$-weakly cofinite.
\end{exam}

Hajkarimi  in \cite[Definition 2.1]{Haj} introduced the class of  weakly Artinian modules as below:

\begin{defn}\label{haj}
	An  $R$-module $M$ is said to be weakly Artinian if its injective envelope, can be written as $\E_R(M)=\oplus_{i=1}^{k}\mu^0({\fm}_i,M)\E_R(R/{\fm}_i)$ where ${\fm}_1,\dots,{\fm}_k$ are maximal ideals of $R$.
\end{defn}
By \cite[Lemma 2.3 (a) and (c)]{Haj} the class of weakly Artinian $R$-modules is a Serre subcategory of the category of $R$-modules and an $R$-module $M$ is Artinian if and only if it is weakly Artinian and $\mu^0({\fm},M)$ is finite for all $\fm\in \Ass_R(M)$. 


The following lemma represent the other equivalent condition for a module to be weakly Artinian.
\begin{lem}
	\label{haj1}
	Let $M$ be an  $R$-module. Then the following statements are equivalent:
	\begin{itemize}
		\item[(i)] $M$ is weakly Artinian.
		\item[(ii)] $\Ass_R(M)$ consists of finitely many maximal ideals.
		\item[(iii)] $\Supp_R(M)$ consists of finitely many maximal ideals.
		\item[(iv)] $\Ass_R(M)=\Supp_R(M)$ and it  consists of finitely many maximal ideals.
		\item[(v)] $M$ is weakly Laskerin and $\Ass_R(M)\subseteq \Max(R)$.
		
	\end{itemize}
\end{lem}

\begin{proof}
	See \cite[Lemma 2.3 (b)]{Haj}.
\end{proof}	

\begin{lem}
	\label{max}
	Let $I$ be an ideal of a Noetherian ring $R$ and $M$ be an  $R$-module such that $\Supp_R(M)\subseteq \Max(R)$. Then the
	following statements are equivalent:
	\begin{itemize}
		\item[(i)] $M$ is $I$-$ETH$-weakly cofinite.
		\item[(ii)] The $R$-module $\Hom_R(R/I,M)$ is weakly Laskerian.
	\end{itemize}
\end{lem}

\begin{proof}
	{\rm(i)}$\Longrightarrow${\rm(ii)}  follows by definition.
	
In order to prove {\rm(ii)}$\Longrightarrow${\rm(i)} note that

\begin{center}
	$ \Hom_R(R/I,\G_I(M))\cong \Hom_R(R/I,M).$
\end{center}

  Since $\Supp_R(\G_I(M))\subseteq \Max(R)$, it is easy to see that $\Hom_R(R/I,\G_I(M))$ is a weakly Artinian $R$-module and so by  \cite[Lemma 2.8]{Haj} $\G_I(M)$ is also a weakly Artinian $R$-module. On the other hand by \cite[Theorem 6.1.2]{BSh} $\lc^{i}_{I}(M)=0$ for all $i\geq 1$. Since any weakly Artinian  $R$-module is weakly Laskerian, therefore
  $\lc^{i}_{I}(M)$ is $I$-weakly cofinite for all $i\geq 0$. Now by  \cite[Corollary 3.10]{Mel},
  it follows that $\Ext^{i}_R(R/I,M)$ are weakly Laskerian for all $i\geq 0$, as required.
\end{proof}

The following Lemma is well-known for $I$-cofinite modules.


\begin{lem}\label{third}
	If  $0\lo N\lo L\lo T\lo 0$ is exact and two of the modules in the sequence are $I$-$ETH$-weakly cofinite, then so is the third
	one.
\end{lem}

\begin{thm}
\label{FD0}
 Let $I$ be an ideal of a Noetherian ring $R$ and $M$ be an  ${\rm FD_{\leq0}}$ {\rm{(}}or minimax{\rm{)}} $R$-module.
  Then the
 following statements are equivalent:
 \begin{itemize}
 \item[(i)] $M$ is $I$-$ETH$-weakly cofinite.
\item[(ii)] The $R$-module $\Hom_R(R/I,M)$ is weakly Laskerian.
 \end{itemize}
\end{thm}

\begin{proof}
	By definition there is a finitely generated submodule $N$ of $M$ such that  ${\rm
		dim}(M/N)\leq 0$. Also, the exact sequence
	$$0\rightarrow N \rightarrow M \rightarrow M/N \rightarrow 0,\,\,\,\,\,\,(*)$$
	induces the following exact sequence
	$$0\longrightarrow {\rm Hom}_R(R/I,N)\longrightarrow{\rm
		Hom}_R(R/I,M)\longrightarrow {\rm Hom}_R(R/I,M/N)$$$$\longrightarrow {\rm
		Ext}^1_R(R/I,N).$$ Whence, it follows
		that the $R$-modules $\Hom_R(R/I,M/N)$ is weakly Laskerian. Therefore, in view of Lemma \ref{max}, the $R$-module $M/N$ is
	$I$-$ETH$-weakly cofinite. Now it follows from the exact sequence $(*)$ and Lemma \ref{third} that $M$ is
	$I$-$ETH$-weakly cofinite.
\end{proof}


We are now ready to state and prove the first main theorem of this section. The
following theorem is a generalization of 
\cite[Proposition 3.2]{BNS2}. In fact, we remove $I$-torsion  condition from this theorem. Note that $I$ is not dimension one too.

\begin{lem} \label {asli}
Let $R$ be a Noetherian ring and $I$ be an ideal of $R$ {\rm (}not necessary
dimension one{\rm )}. Let $M$ be a non-zero $R$-module {\rm (}not necessary
$I$-torsion{\rm )} such that  $\dim M \leq 1$. Then the following conditions are
equivalent:
 \begin{itemize}
 \item[(i)] $M$ is $I$-$ETH$-weakly cofinite.
 \item[(ii)] $\lc^{i}_{I}(M)$ are $I$-weakly cofinite for all $i$.
\item[(iii)] The $R$-modules $\Hom_R(R/I,M)$ and $\Ext^{1}_R(R/I,M)$ are weakly Laskerian.
 \end{itemize}
\end{lem}

\begin{proof}
{\rm(iii)}$\Longrightarrow${\rm(ii)} Using the exact sequence
$$0\rightarrow\G_I(M)\rightarrow M\rightarrow M/\G_I(M) \rightarrow 0,$$
we get the exact sequence
\begin{center}
$0\lo \Hom_R(R/I,\G_I(M))\lo \Hom_R(R/I,M)\lo \Hom_R(R/I,M/\G_I(M))\lo
\Ext^{1}_R(R/I,\G_I(M))\lo \Ext^{1}_R(R/I,M).$
\end{center}
Since $\Hom_R(R/I,M/\G_I(M))=0$, it follows that the $R$-modules
\begin{center}
$\Hom_R(R/I,\G_I(M))$ and $\Ext^{1}_R(R/I,\G_I(M))$
\end{center}
are weakly Laskerian, and so in view of  \cite[Proposition 3.2]{BNS2} the
$R$-module $\G_I(M)$ is $I$-weakly cofinite. Now as the $R$-module $
\Ext^{1}_R(R/I,M)$ is weakly Laskerian, it follows from \cite[Theorem 4.1 (c)]{AM}
that the $R$-module $\Hom_R(R/I,\lc^{1}_{I}(M))$ is weakly Laskeran. If $\fp \in \Supp_R(\lc^{1}_{I}(M))\subseteq \Supp_R(M)$, then 
$$\lc^{1}_{IR_{\fp}}(M_\fp)\cong\lc^{1}_{I}(M)_\fp\neq 0.$$

Since $\dim M \leq 1$, it is easy to see that $\dim R/\fp=0$ or $\dim R/\fp=1$.  If $\dim R/\fp=1$ then $ M_\fp$ is a zero dimensional $R_\fp$-module that implies $\lc^{1}_{IR_{\fp}}(M_\fp)=0$  by using Grothendieck vanishing theorem \cite[Theorem 6.1.2]{BSh} which is a contradiction. Thus $\dim R/\fp=0$ and so $\fp$ is a maximal ideal. So we have the following inclusion 
$$\Supp_R(\Hom_R(R/I,\lc^{1}_{I}(M)))\subseteq \Supp_R(\lc^{1}_{I}(M))\subseteq \Max R.$$
By Lemma \ref{haj1} (v), it is easy to see that the $R$-module $\Hom_R(R/I,\lc^{1}_{I}(M))$ is weakly Artinian and
so by \cite[Lemma 2.8]{Haj} the $R$-module $\lc^{1}_{I}(M)$ is weakly Artinian. Since any weakly Artinian module is weakly Laskerian, therefore in view of  \cite[Theorem 6.1.2]{BSh}  the $R$-module
$\lc^{i}_{I}(M)$ is $I$-weakly cofinite for all $i\geq 0$.

{\rm(i)}$\Longrightarrow${\rm(ii)}  by  \cite[Corollary 3.10]{Mel},
it follows that $\Ext^{i}_R(R/I,M)$ are weakly Artinian for all $i\geq 0$, as required.

{\rm(i)}$\Longrightarrow${\rm(iii)} It is obviously true.
\end{proof}

The following theorem is a generalization of \cite[Theorem 3.1]{AB} that in what
follows the next theorem plays an important role.

\begin{thm} \label {wl}
Let $R$ be a Noetherian ring and $I$ be an ideal of $R$. Let $M$ be an ${\rm
FD_{\leq 1}}$  $R$-module. Then $M$ is $I$-$ETH$-weakly cofinite if and only if
 $\Hom_R(R/I,M)$ and  $\Ext^{1}_R(R/I,M)$ are weakly Lakerian.
\end{thm}

\begin{proof}
By definition there is a finitely generated submodule $N$ of $M$ such that  ${\rm
dim}(M/N)\leq 1$. Also, the exact sequence
$$0\rightarrow N \rightarrow M \rightarrow M/N \rightarrow 0,\,\,\,\,\,\,(*)$$
induces the following exact sequence
$$0\longrightarrow {\rm Hom}_R(R/I,N)\longrightarrow{\rm
Hom}_R(R/I,M)\longrightarrow {\rm Hom}_R(R/I,M/N)$$$$\longrightarrow {\rm
Ext}^1_R(R/I,N)\longrightarrow{\rm Ext}^1_R(R/I,M)\longrightarrow {\rm
Ext}^1_R(R/I,M/N)\longrightarrow {\rm Ext}^2_R(R/I,N).$$ Whence, it follows that
the $R$-modules $\Hom_R(R/I,M/N)$ and $\Ext^1_R(R/I,M/N)$ are weakly Laskerian. Therefore, in view of Proposition \ref{asli}, the $R$-module $M/N$ is
 $I$-$ETH$-weakly cofinite. Now it follows from the exact sequence $(*)$  and Lemma \ref{third} that $M$ is
 $I$-$ETH$-weakly cofinite.
\end{proof}




The following lemma is needed in the proof of second main result of this paper.

\begin{lem}
	\label{2.2}
	Let $I$ be an ideal of a Noetherian ring $R$, $M$ a non-zero $R$-module and
	$t\in\Bbb{N}_0$. Suppose that the $R$-module $\lc^{i}_\Phi(M)$ is $I$-$ETH$-weakly cofinite
	for all $i=0,...,t-1$, and the $R$-modules ${\rm Ext}^{t}_R(R/I,M)$ and ${\rm
		Ext}^{t+1}_R(R/I,M)$ are weakly Laskerian. Then the $R$-modules ${\rm
		Hom}_R(R/I,\lc^{t}_\Phi(M))$ and ${\rm Ext}^{1}_R(R/I,\lc^{t}_\Phi(M))$ are weakly Laskerian.
\end{lem}

\begin{proof}
 We use induction on $t$. The exact sequence 
	
	$$0\lo \G_\Phi(M)\lo M \lo M/\G_\Phi(M)
	\lo 0,\,\,\,\,\,\,(*)$$
	
	induces the following exact sequence:
	
	$$0\longrightarrow {\Hom}_R(R/I,\G_\Phi(M))\lo{\Hom}_R(R/I,M)\lo {\Hom}_R(R/I,M/\G_\Phi(M))$$$$\lo {\Ext}^1_R(R/I,\G_\Phi(M))\lo{\Ext}^1_R(R/I,M).$$
	
	Since ${\Hom}_R(R/I,M/\G_\Phi(M))=0$ so ${\rm
		Hom}_R(R/I,\G_\Phi(M))$ and ${\rm Ext}^{1}_R(R/I,\G_\Phi(M))$ are weakly Laskerian. Assume inductively that $t>0$ and that we have established the result for non-negative integers smaller than $t$. By applying the functor ${\Hom}_R(R/I,-)$  to the exact sequence $(*)$, we can deduce that ${\Ext}_R^j(R/I,M/\G_\Phi(M))$
	is weakly Laskerian for $j=t,t+1$. On the other hand, $\lc_I^0(M/\G_\Phi(M))=0$ and $\lc_\Phi^j(M/\G_\Phi(M))\cong \lc_\Phi^j(M)$ for all $j>0$. Therefore we may assume that $\G_\Phi(M)=0$. Let $E$ be an injective hull of $M$ and put $N=E/M$. Then ${\Hom}_R(R/I,E)=0=\G_\Phi(E)$. Hence  ${\Ext}_R^j(R/I,N)\cong {\Ext}_R^{j+1}(R/I,M)$ and $\lc_\Phi^j(N)\cong \lc_\Phi^{j+1}(M)$ for all $j\geq 0$. Now, the induction hypothesis yields that ${\rm
		Hom}_R(R/I,\lc_\Phi^{t-1}(N))$ and ${\rm Ext}^{1}_R(R/I,\lc_\Phi^{t-1}(N))$ are weakly Laskerian and so ${\rm
		Hom}_R(R/I,\lc_\Phi^{t}(M))$ and ${\rm Ext}^{1}_R(R/I,\lc_\Phi^{t}(M))$ are weakly Laskerian, as required.
\end{proof}

We are now ready to state and prove the following main results (Theorem
\ref{homext} and  the Corollaries \ref{homext1}, \ref{cof}, \ref{cof1}, \ref{cof2}) which are
extension of \cite[Theorem 3.4 and Corollaries 3.5 and 3.6]{AB}, Bahmanpour-Naghipour's results in
\cite{BN, BN2} in terms of weakly Laskerian modules, Hong Quy's result in
\cite{hung}, Divaani-Aazar and Mafi's result in  \cite[Corrolary  2.7]{DiM}, \cite[Theorem 2.13]{Z} and \cite[Corollaries 2.6 and 2.7]{NT}.

\begin{thm}
	\label{homext} Let $R$ be a Noetherian ring and $I\in \Phi$ an ideal of $R$. Let $t\in\Bbb{N}_0$ be an integer and $M$ an $R$-module such that $\Ext^i_R(R/I,M)$ are weakly Laskerian for all $i\leq t+1$. Let the $R$-modules $\lc^{i}_\Phi(M)$ are ${\rm FD_{\leq 1}}$ $R$-modules for all $i<t$. Then, the following conditions hold:
	\begin{itemize}
		\item[(i)]  The $R$-modules $\lc^{i}_\Phi(M)$ are $I$-$ETH$-weakly cofinite {\rm{(}}in particular $\Phi$-weakly cofinite{\rm{)}} for all $i<t$.
		\item[(ii)]  For all ${\rm FD_{\leq 0}}$ {\rm{(}}or minimax{\rm{)}} submodule $N$ of
		$\lc^{t}_\Phi(M)$, the $R$-modules $${\rm Hom}_R(R/I,\lc^{t}_\Phi(M)/N)\,\,\,
		{\rm and}\,\,\,{\rm Ext}^1_R(R/I,\lc^{t}_\Phi(M)/N)$$ are weakly Laskerian. In
		particular the sets $$\Ass_R({\rm Hom}_R(R/I,\lc^{t}_\Phi(M)/N))\,\,\, {\rm
			and}\,\,\,\Ass_R({\rm Ext}^1_R(R/I,\lc^{t}_\Phi(M)/N))$$ are  finite sets.
	\end{itemize}
\end{thm}
\proof (i) We proceed by induction on $ t$. In the case $ t = 0$ there is nothing to prove. So, let $ t >0 $ and the
result has been proved for smaller values of $t$. By the inductive assumption,
$\lc^{i}_\Phi(M)$ is $I$-$ETH$-weakly cofinite for $ i = 0, 1, ..., t-2$. Hence by Lemma
\ref{2.2} and assumption, ${\rm Hom}_R(R/I,\lc^{t-1}_\Phi(M))\,\,\, {\rm
	and}\,\,\,{\rm Ext}^1_R(R/I,\lc^{t-1}_\Phi(M))$ are weakly Laskerian. Therefore by
Theorem
\ref{wl}, $\lc^{i}_\Phi(M)$ is $I$-$ETH$-weakly cofinite {\rm{(}}in particular, $\Phi$-weakly cofinite{\rm{)}} for all $ i < t$. This completes the inductive step.\\
(ii) In view of   (i) and Lemma \ref{2.2},  ${\rm Hom}_R(R/I,\lc^{t}_\Phi(M))\,\,\, {\rm
	and}\,\,\,{\rm Ext}^1_R(R/I,\lc^{t}_\Phi(M))$ are weakly Laskerian. On the other
hand, according to Lemma \ref{FD0}, $N$ is $I$-$ETH$-weakly cofinite. Now, the exact
sequence

$$0\longrightarrow N\longrightarrow \lc^{t}_\Phi(M) \longrightarrow \lc^{t}_\Phi(M)/N
\longrightarrow 0$$

induces the following exact sequence,
$${\rm
	Hom}_R(R/I,\lc^{t}_\Phi(M))\longrightarrow {\rm
	Hom}_R(R/I,\lc^{t}_\Phi(M)/N)\longrightarrow {\rm Ext}^1_R(R/I,N)\longrightarrow$$
$$ {\rm Ext}^1_R(R/I,\lc^{t}_\Phi(M))\longrightarrow {\rm
	Ext}^1_R(R/I,\lc^{t}_\Phi(M)/N)\longrightarrow {\rm Ext}^2_R(R/I,N).$$

Consequently  $${\rm Hom}_R(R/I,\lc^{t}_\Phi(M)/N)\,\,\, {\rm and}\,\,\,{\rm
	Ext}^1_R(R/I,\lc^{t}_\Phi(M)/N)$$ are weakly Laskerian, as required.\qed\\

\begin{cor}
	\label{homext1} Let $R$ be a Noetherian ring and $I$ an ideal of $R$. Let $t\in\Bbb{N}_0$ be an integer and $M$ an $R$-module such that $\Ext^i_R(R/I,M)$ are weakly Laskerian for all $i\leq t+1$. Let the $R$-modules $H^i_I(M)$ are ${\rm FD_{\leq1}}$ ~ $R$-modules for all $i<t$.
	Then, the following conditions hold:
	
	{\rm(i)} The $R$-modules $H^i_I(M)$ are $I$-weakly cofinite for all
	$i<t$.
	
	{\rm(ii)} For all ${\rm FD_{\leq 0}}$ {\rm{(}}or minimax{\rm{)}} submodule $N$ of
	$H^{t}_I(M)$, the $R$-modules $${\rm Hom}_R(R/I,H^{t}_I(M)/N)\,\,\, {\rm
		and}\,\,\,{\rm Ext}^1_R(R/I,H^{t}_I(M)/N)$$ are weakly Laskerian. In particular the
	set $\Ass_R(H^{t}_I(M)/N)$ is  finite.
\end{cor}


The following corollaries answer to Hartshorne's question (i.e., Question 1.2) in terms of weakly cofiniteness.

\begin{cor}
	\label{cof} Let $R$ be a Noetherian ring and $I\in \Phi$ an ideal of $R$. Let $M$
	be an $I$-$ETH$-weakly cofinite $R$-module  such that the $R$-modules $\lc^{i}_\Phi(M)$
	are ${\rm FD_{\leq 1}}$ $R$-modules for all
	$i$. Then,
	\begin{itemize}
		\item[(i)] the $R$-modules $\lc^{i}_\Phi(M)$ are $I$-$ETH$-weakly cofinite {\rm{(}}in particular, $\Phi$-weakly cofinite{\rm{)}} for all $i$.
		\item[(ii)]  for any $i\geq 0$ and for any ${\rm FD_{\leq 0}}$ {\rm{(}}or minimax{\rm{)}} submodule $N$ of
		$\lc^{i}_\Phi(M)$, the $R$-module $\lc^{i}_\Phi(M)/N$ is   $I$-$ETH$-weakly cofinite
		{\rm{(}}in particular, $\Phi$-weakly cofinite{\rm{)}}.
	\end{itemize}
\end{cor}
\proof (i)  Clear.\\ (ii)  In view of  (i) the $R$-module $\lc^{i}_\Phi(M)$ is
$I$-$ETH$-weakly cofinite for all $i$. Hence the $R$-module ${\rm Hom}_R(R/I,N)$ is
weakly Laskerian, and so it follows from Lemma \ref{FD0} that $N$ is
$I$-$ETH$-weakly cofinite. Now, the exact sequence

$$0\longrightarrow N\longrightarrow \lc^{i}_\Phi(M)\longrightarrow \lc^{i}_\Phi(M)/N
\longrightarrow 0,$$

and Lemma \ref{third} implies that the $R$-module $\lc^{i}_\Phi(M)/N$ is $I$-$ETH$-weakly cofinite.\qed\\


\begin{cor}
	\label{cof1} Let $R$ be a Noetherian ring and $I$ an ideal of $R$. Let $M$
	be an $I$-$ETH$-weakly cofinite $R$-module  such that $R$-modules $H^i_I(M)$ are ${\rm
		FD_{\leq1}}$  $R$-modules for all $i$. Then,
	
	{\rm(i)} The $R$-modules $H^i_I(M)$ are $I$-weakly cofinite for all $i$.
	
	{\rm(ii)} For any $i\geq 0$ and for any ${\rm FD_{\leq 0}}$ {\rm{(}}or minimax{\rm{)}} submodule $N$ of
	$H^{i}_I(M)$, the $R$-module $H^{i}_I(M)/N$ is  $I$-weakly cofinite.
\end{cor}


\begin{cor}
	Let $R$ be a Noetherian ring and $I$ an ideal of $R$. Let $M$ be
	an $R$-module such that  the $R$-modules $H^i_I(M)$ are ${\rm
		FD_{\leq1}}$ $R$-modules for all $i$. Then, the following conditions are equivalent:
	
	{\rm(i)} The $R$-module $M$
	is an $I$-$ETH$-weakly cofinite.
	
	{\rm(ii)} The $R$-modules $H^i_I(M)$ are $I$-weakly cofinite for all $i$.	
\end{cor}

\proof (i)$\Rightarrow$(ii) Follows by  Corollary \ref{cof1}.

(ii)$\Rightarrow$(i) It follows by \cite[Proposition 3.9]{Mel}.\qed\\


The following corollary is a generalization of \cite[Corollary 2.7]{BN} in terms of weakly cofiniteness.

\begin{cor}
	\label{cof2} Let $R$ be a Noetherian ring and $I\in \Phi$ an ideal of $R$. Let
	$M$ be an $I$-$ETH$-weakly cofinite $R$-module  such that  $\dim M/IM \leq 1$ {\rm(}e.g.,
	$\dim R/I \leq 1${\rm)} for all $I \in\Phi$. Then,
	\begin{itemize}
		\item[(i)] the $R$-modules $\lc^{i}_\Phi(M)$ are $I$-$ETH$-weakly cofinite {\rm{(}}in particular, $\Phi$-weakly cofinite{\rm{)}} for all $i$.
		\item[(ii)]  for any $i\geq 0$ and for any ${\rm FD_{\leq 0}}$ {\rm{(}}or minimax{\rm{)}} submodule $N$ of
		$\lc^{i}_\Phi(M)$, the $R$-module $\lc^{i}_\Phi(M)/N$ is   $I$-$ETH$-weakly cofinite
		{\rm{(}}in particular, $\Phi$-weakly cofinite{\rm{)}}.
	\end{itemize}
\end{cor}
\begin{proof}
	(i) Since  by \cite[Lemma 2.1]{BZ1}, $$\lc^{i}_\Phi(M)\cong\underset { I \in \Phi}
	\varinjlim \lc^{i}_I(M),$$ it is easy to see that $\Supp_R(\lc^{i}_\Phi(M))
	\subseteq\underset{I \in \Phi}\bigcup \Supp_R(\lc^{i}_I(M))$ and therefore
	$$\dimSupp \lc^{i}_\Phi(M) \leq \sup \{ \dimSupp \lc^{i}_{I}(M)| I \in
	\Phi \} \leq 1,$$ thus $\lc^{i}_\Phi(M)$ is ${\rm FD_{\leq 1}}$ $R$-module
	and the assertion follows by  Corollary \ref{cof}
	(i).\\
	(ii) Proof is the same as \ref{cof} (ii).
\end{proof}


\begin{cor}
	Let $R$ be a Noetherian ring and $I$ an ideal of $R$. Let $M$ be
	an $R$-module such that  $\dim M/IM \leq 1$ {\rm(}e.g.,
	$\dim R/I \leq 1${\rm)}. Then, the following conditions are equivalent:
	
	{\rm(i)} The $R$-modules $M$
	is an $I$-$ETH$-weakly cofinite.
	
	{\rm(ii)} The $R$-modules $H^i_I(M)$ are $I$-weakly cofinite for all $i$.	
\end{cor}

\proof (i)$\Rightarrow$(ii) Follows by  Corollary \ref{cof2}.

(ii)$\Rightarrow$(i) It follows by \cite[Proposition 3.9]{Mel}.\qed\\


One of the main results of this section is to prove that for an arbitrary ideal $I$ of
a Noetherian ring $R$, the category
 of  $\mathscr {FD}^1(R,I)_{ethwcof}$ modules is Abelian category.

\begin{thm}
\label{abel1}
 Let $I$ be an ideal of a Noetherian ring $R$. Let   $\mathscr {FD}^1(R,I)_{ethwcof}$  denote the category of $I$-$ETH$-weakly cofinite and ${\rm FD_{\leq1}}$ ~ $R$-modules. Then
 $\mathscr {FD}^1(R,I)_{ethwcof}$
is an Abelian category.
\end{thm}

\begin{proof}
 Let $M,N\in \mathscr {FD}^1(R,I)_{ethwcof}$ and let $f:M\longrightarrow N$ be an
$R$-homomorphism. Since by \cite[Lemma 2.3 (v)]{AB} the class of ${\rm FD_{\leq 1}}$ is a Serre subcategory of the category of $R$-modules, it is enough  to show that the $R$-modules $\Ker f$ and
$\Coker f$ are $I$-$ETH$-weakly cofinite.

To this end, the exact sequence
$$0\longrightarrow \Ker f \longrightarrow M \longrightarrow
\Image f \longrightarrow 0,$$ induces an  exact sequence
$$0\longrightarrow {\rm Hom}_R(R/I,\Ker f) \longrightarrow {\rm
Hom}_R(R/I,M) \longrightarrow {\rm Hom}_R(R/I,\Image f )$$$$\longrightarrow
{\rm Ext}^{1}_R(R/I,\Ker f) \longrightarrow \Ext^{1}_R(R/I,M),$$ that implies the
$R$-modules $\Hom_R(R/I,\Ker f)$ and $\Ext^1_R(R/I,\Ker f)$ are weakly Laskerian. Since $\Ker f$ is  ${\rm
	FD_{\leq1}}$~$R$-module, therefore it follows from Theorem \ref{wl} that $\Ker f$ is
 $I$-$ETH$-weakly cofinite. Now, the assertion follows from the following exact
sequences
$$0\longrightarrow \Ker f \longrightarrow M \longrightarrow
\Image f  \longrightarrow 0,$$and
$$0\longrightarrow  \Image f \longrightarrow N \longrightarrow
\Coker f \longrightarrow 0.$$
\end{proof}


The following corollaries are generalization of \cite[Corolaries  3.8 and 3.9]{AB}.

\begin{cor}
\label{torext}
 Let $R$ be a Noetherian ring and $I$ a proper ideal of $R$. Let  $M$ is a non-zero $I$-$ETH$-weakly cofinite  ${\rm FD_{\leq1}}$ ~ $R$-module. Then, the R-modules ${\rm Ext}^{i}_R(N,M)$ and ${\rm Tor}^R_{i}(N,M)$ are $I$-$ETH$-weakly cofinite  ${\rm
FD_{\leq1}}$~ $R$-modules, for all finitely generated $R$-modules $N$ and all integers $ i \geq 0$.
\end{cor}
\proof Since $N$ is finitely generated it follows that $N$ has a free resolution of
finitely generated free modules. Now the assertion follows using Theorem
\ref{abel1} and computing the modules ${\rm Tor}_i^R(N,M)$ and ${\rm
Ext}^i_R(N,M)$, by this
free resolution. \qed\\

\begin{cor}
\label{torextlc} Let $\Phi$ be a system of ideals of  a Noetherian ring $R$, $M$ a non-zero
$I$-$ETH$-weakly cofinite $R$-module such that $ \lc^{i}_{\Phi}(M)$ are ${\rm FD_{\leq 1}}$~ $R$-modules for all $ i \geq 0$. Then for each finite
$R$-module $N$, the $R$-modules ${\rm Ext}^{j}_R(N,\lc^{i}_{\Phi}(M))$ and ${\rm
Tor}^R_{j}(N,\lc^{i}_{\Phi}(M))$ are $\Phi$-weakly cofinite and ${\rm FD_{\leq 1}}$~ $R$-modules for all $ i \geq 0$ and $ j \geq 0$.
\end{cor}
\proof Apply Corollaries  \ref{torext} and \ref{cof}.\qed\\

\begin{cor}
	\label{torextlc1} Let $\Phi$ be a system of ideals of  a Noetherian ring $R$, $M$ a non-zero
	$I$-$ETH$-weakly cofinite $R$-module such that $\dim M/IM \leq 1$ {\rm(}e.g.,
	$\dim R/I \leq 1${\rm)} for all $I \in\Phi$. Then for each finite
	$R$-module $N$, the $R$-modules ${\rm Ext}^{j}_R(N,\lc^{i}_{\Phi}(M))$ and ${\rm
		Tor}^R_{j}(N,\lc^{i}_{\Phi}(M))$ are $\Phi$-weakly cofinite and ${\rm FD_{\leq 1}}$~ $R$-modules for all $ i \geq 0$ and $ j \geq 0$.
\end{cor}
\proof By proof of Corollary \ref{cof2} (i) $\Supp_R(\lc^{i}_\Phi(M))
\subseteq\underset{I \in \Phi}\bigcup \Supp_R(\lc^{i}_I(M))$ and therefore
$$\dimSupp \lc^{i}_\Phi(M) \leq \sup \{ \dimSupp \lc^{i}_{I}(M)| I \in
\Phi \} \leq 1,$$ thus $\lc^{i}_\Phi(M)$ is ${\rm FD_{\leq 1}}$ $R$-module
and the assertion follows by  Corollary \ref{torextlc}.\qed\\


\begin{cor}
	\label{torextlc2} Let $I$ be an ideal of  a Noetherian ring $R$, $M$ a non-zero
	$I$-$ETH$-weakly cofinite $R$-module such that $\dim M/IM \leq 1$ {\rm(}e.g.,
	$\dim R/I \leq 1${\rm)}. Then for each finite
	$R$-module $N$, the $R$-modules ${\rm Ext}^{j}_R(N,\lc^{i}_{I}(M))$ and ${\rm
		Tor}^R_{j}(N,\lc^{i}_{I}(M))$ are $I$-weakly cofinite and ${\rm FD_{\leq 1}}$~ $R$-modules for all $ i \geq 0$ and $ j \geq 0$.
\end{cor}


\section{Weakly cofiniteness of local cohomology defined by a pair of ideals}

 As a special case of general local cohomology and generalization of ordinary  local cohomology modules, R. Takahashi, Y. Yoshino, and T.
 Yoshizawa \cite{TYY}, introduced local cohomology modules with respect to a
 pair of ideals. The $(I,J)$-torsion submodule $\Gamma_{I,J}(M)$ of $M$ is a
 submodule of $M$ consists of all elements $x$ of $M$ with Supp$(Rx)\subseteq
 W(I,J)$, in which
 $$W(I,J)=\{~\fp\in\textmd{Spec}(R)\mid I^n\subseteq
 \fp+J \textmd{\ \ for an integer} \ n\geq1\}.$$
 
 For an integer $i$, the $i$-th local cohomology functor $\lc^{i}_{I,J}$ with respect
 to $(I,J)$ is the $i$-th right derived functor of $\Gamma_{I,J}$. The $R$-module
 $\lc^{i}_{I,J}(M)$ is called the $i$-th local cohomology module of $M$ with
 respect to $(I,J)$. In the case $J=0$, $\lc^{i}_{I,J}(-)$ coincides with the ordinary
 local cohomology functor $\lc^{i}_{I}(-)$. Also, we are concerned with the
 following set of ideals of $R$:
 $$\tilde{W}(I,J)=\{~\fa\trianglelefteq R\mid
 I^n\subseteq \fa+J \textmd{\ \ for an integer} \ n\geq0\}.$$
 
The definition of  weakly cofinite module with respect to an ideal   {\rm{(}}\cite[Definition 2.4]{DiM2}{\rm{)}}
motivates the following definition.

\begin{defn}
	An $R$-module $M$ is called  $(I,J)$-weakly cofinite if
	$\Supp_R(M)\subseteq W(I,J)$ and $\Ext^i_R(R/I,M)$ is a weakly Laskerian $R$-module, for
	all $i\geq0$.
\end{defn} 
 

\begin{rem} Let $I$ and $J$ be two ideals of $R$. Replacing $\Phi$ by $\tilde{W}(I,J)$, $\lc^{i}_\Phi(M)$ by $\lc^{i}_{I,J}(M)$ and $\Phi$-weakly cofinite module by $(I,J)$-weakly cofinite module, the Theorem \ref{homext} and  Corollaries \ref{cof}, \ref{cof2}, \ref{torextlc} and \ref{torextlc1} are true for local cohomology modules defined by a
	pair of ideals. Because by  \cite[{\it Definition 3.1 and Theorem 3.2}\rm]{TYY}, it is easy
	to see that the local cohomology modules defined by a pair of ideals is a special
	case of local cohomology modules with respect to a system of ideals.
\end{rem}

%




\bibliographystyle{amsplain}

\end{document}